\documentclass[reqno]{amsart}

\usepackage{amsfonts,amssymb,graphicx,float}
\usepackage[arrow,matrix,curve]{xy}

\restylefloat{figure}

\makeatletter                                                  %
\def\th@plain{\slshape}                                        %
\makeatother                                                   %

\newcommand{\oi}{[0,1]}

\newcommand{\Nbb}{\mathbb{N}}
\newcommand{\Zbb}{\mathbb{Z}}
\newcommand{\Qbb}{\mathbb{Q}}
\newcommand{\Rbb}{\mathbb{R}}

\newcommand{\Rp}{\mathbb{R}_{\ge0}}

\newcommand{\monk}{M\"onkemeyer}

\newcommand{\dyadic}{\Zbb[1/2]}

\newcommand{\Bcal}{\mathcal{B}}
\newcommand{\Fcal}{\mathcal{F}}
\newcommand{\stochastic}{\mathfrak{K}}

\newcommand{\Ccal}{\mathcal{C}}
\newcommand{\Dcal}{\mathcal{D}}
\newcommand{\Scal}{\mathcal{S}}
\newcommand{\Tcal}{\mathcal{T}}
\newcommand{\Vcal}{\mathcal{V}}
\newcommand{\abar}{\bar a}
\newcommand{\bbar}{\bar b}
\newcommand{\Btilde}{\tilde B}
\newcommand{\cantor}{\{0,1\}^\Nbb}

\newcommand{\lland}{\:\text{\&}\:}

\newcommand{\newword}[1]{\textsl{#1}}
\newcommand{\pp}[1]{\mathbf{#1}}
\newcommand{\vect}[3]{#1_#2,\ldots ,#1_#3}

\newcommand{\abs}[1]{\lvert#1\rvert}

\newcommand{\angles}[1]{\langle #1 \rangle}

\DeclareMathSymbol{\upharpoonright}{\mathrel}{AMSa}{"16}
\let\restriction\upharpoonright
\DeclareMathSymbol{\nmid}{\mathrel}{AMSb}{"2D}

\DeclareMathOperator{\mesh}{mesh}
\DeclareMathOperator{\GO}{GO}
\DeclareMathOperator{\EPO}{EPO}

\DeclareMathOperator{\GL}{GL}

\DeclareMathOperator{\PSL}{PSL}

\DeclareMathOperator{\Id}{Id}
\DeclareMathOperator{\Mat}{Mat}

\theoremstyle{plain}
\newtheorem{theorem}{Theorem}[section]
\newtheorem{lemma}[theorem]{Lemma}
\newtheorem{proposition}[theorem]{Proposition}
\newtheorem{corollary}[theorem]{Corollary}

\theoremstyle{definition}
\newtheorem{definition}[theorem]{Definition}

\begin{document}

\bibliographystyle{plain}

\sloppy

\title[An $n$-dimensional Minkowski function]{Multidimensional continued fractions\\
and a Minkowski function}

\author[G. Panti]{Giovanni Panti}
\address{Department of Mathematics\\
University of Udine\\
via delle Scienze 208\\
33100 Udine, Italy}
\email{panti@dimi.uniud.it}

\begin{abstract}
The Minkowski Question Mark function can be characterized as the unique homeomorphism of the real unit interval that conjugates the Farey map with the tent map. We construct an $n$-dimensional analogue of the Minkowski function as the only homeomorphism of an $n$-simplex that conjugates the piecewise-fractional map associated to the \monk{} continued fraction algorithm with an appropriate tent map.
\end{abstract}


\thanks{\emph{2000 Math.~Subj.~Class.}: 11J70; 37A45}

\maketitle

\section{Preliminaries}\label{ref2}

The $n$th order \newword{Farey set} $\Fcal_n$ in the real unit interval $\oi$
is defined by recursion: one starts with $\Fcal_0=\{0/1,1/1\}$ and obtains $\Fcal_n$ by adding to $\Fcal_{n-1}$ all the \newword{Farey sums} $v_1\oplus v_2=(a_1+a_2)/(b_1+b_2)$ of two consecutive elements $v_i=a_i/b_i$ of $\Fcal_{n-1}$. The union of all the $\Fcal_n$'s is the set of all rational numbers in $\oi$. Analogously, by starting with $\Bcal_0=\Fcal_0$ and replacing the Farey sum with the \newword{barycentric sum} $v_1\boxplus v_2=(v_1+v_2)/2$, we obtain an increasing sequence $\Bcal_0\subset\Bcal_1\subset\Bcal_2\subset\cdots$, whose union is the set of all dyadic rationals in $\oi$. For every $n\ge0$, there exists a unique order-preserving bijection from $\Fcal_n$ to $\Bcal_n$. The union of these bijections is a bijection from $\bigcup_{n\ge0}\Fcal_n$ to $\bigcup_{n\ge0}\Bcal_n$, which extends uniquely by continuity to an order-preserving bijection $\Phi:\oi\to\oi$. This last map is the Minkowski Question Mark function~\cite{denjoy38}, \cite{salem43}, \cite{kinney60}, \cite{viaderparjabi98}. Among others, $\Phi$ has the following properties:
\begin{enumerate}
\item it is an order-preserving homeomorphism of $\oi$;
\item it maps bijectively the rational numbers to the dyadic rationals, and the real algebraic numbers of degree $\le2$ to the rationals (all these sets restricted to $\oi$, of course);
\item it is singular w.r.t.~the Lebesgue measure $\lambda$ (i.e., there exists a measurable set $A\subseteq\oi$ such that $\lambda(A)=1$ and $\lambda\bigl(\Phi[A]\bigr)=0$);
\item it has a fractal structure ---which is apparent in the following approximate sketch---
\begin{figure}[H]
\begin{center}
\includegraphics[height=3.5cm,width=3.5cm]{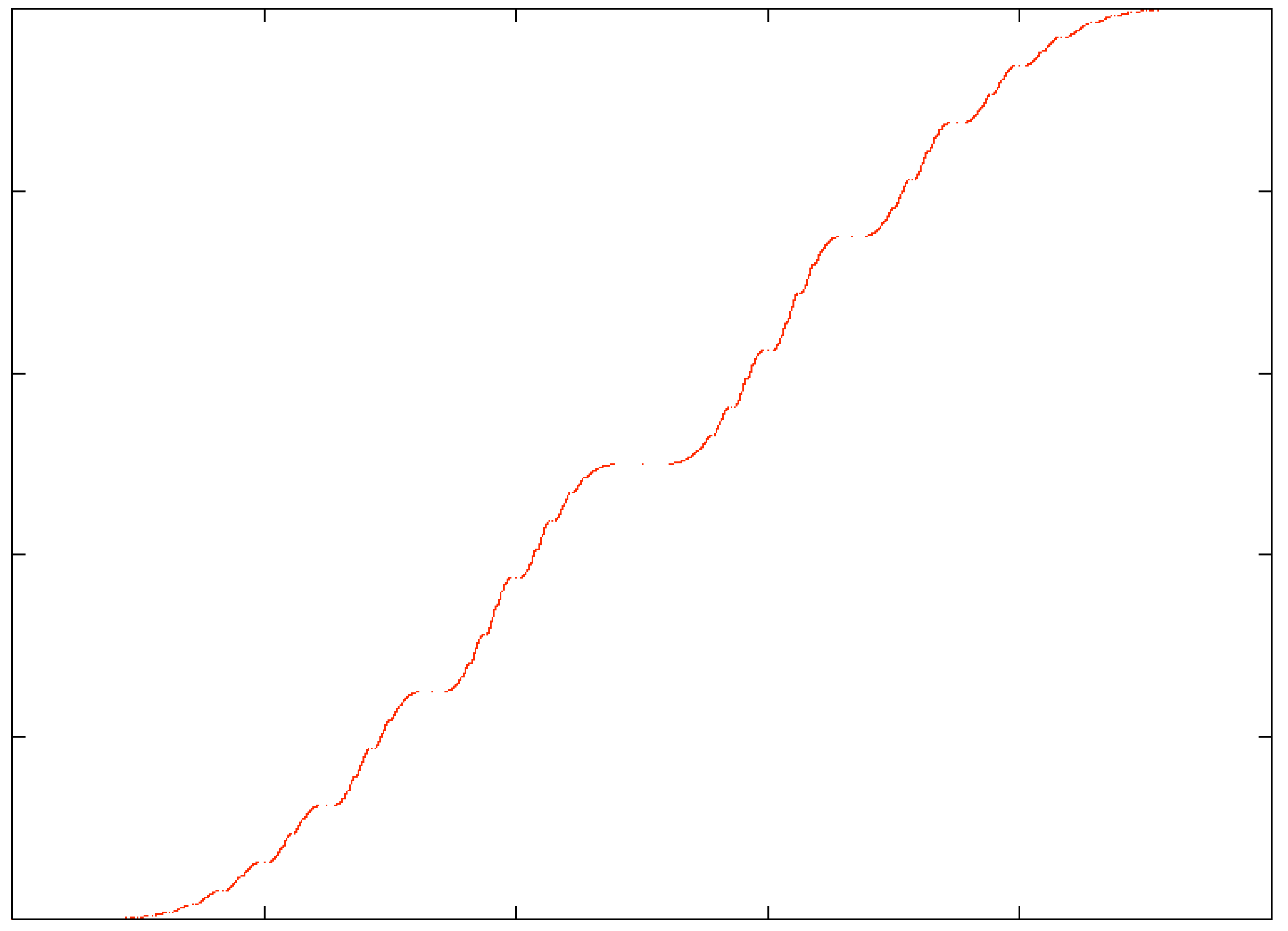}
\end{center}
\end{figure}
\item it conjugates the \newword{Farey map}
$$
F(x)=\begin{cases}
x/(1-x), & \text{if $0\le x<1/2$;}\\
(1-x)/x, & \text{if $1/2\le x\le 1$;}
\end{cases}
$$
with the \newword{tent map}
$$
T(x)=\begin{cases}
2x, & \text{if $0\le x<1/2$;}\\
2-2x, & \text{if $1/2\le x\le 1$.}
\end{cases}
$$
\end{enumerate}

Property~(4) means, more precisely, the following: let $v_1<v_2$ be consecutive elements of some $\Fcal_n$. Then there exists a unique element
$\bigl( \begin{smallmatrix} 
a&b\\ c&d 
\end{smallmatrix} \bigr)$
of $\PSL_2\Zbb$ such that the corresponding fractional-linear transformation $G(x)=(ax+b)/(cx+d)$ maps $v_1$ to $0$ and $v_2$ to $1$. Analogously, there is a unique affine transformation $H(x)=rx+s$ such that $H(\Phi(v_1))=0$ and $H(\Phi(v_2))=1$. One then checks easily that $\Phi=H\circ(\Phi\restriction[v_1,v_2])\circ G^{-1}$.

We note the following for future reference.

\begin{proposition}
Property~(5)\label{ref3} characterizes $\Phi$.
\end{proposition}
\begin{proof}
Let $\Psi$ be a homeomorphism of $\oi$ such that $T=\Psi\circ F\circ\Psi^{-1}$. The only point which is fixed by $F$ (respectively, $T$), and whose removal does not disconnect $\oi$ is $0$; therefore $\Psi(0)=0$ and $\Psi$ is order-preserving. For every $n\ge0$ we have $\Fcal_n=F^{-(n+1)}\{0\}$ and $\Bcal_n=T^{-(n+1)}\{0\}$. Hence, for every $n$, $\Psi$ restricts to a bijection between $\Fcal_n$ and $\Bcal_n$. Since these bijections are order-preserving, $\Psi$ must coincide with $\Phi$.
\end{proof}

In~\cite{beavergar04} a generalization of the Minkowski function to a selfmap $\delta$ of a $2$-dimensional triangle is proposed. The construction of $\delta$ proceeds in stages, and parallels that for $\Phi$: assume that $\angles{v^1_1,v^1_2,v^1_3}$ and $\angles{v^2_1,v^2_2,v^2_3}$ are paired triangles that appear at the $(n-1)$th stage of the construction ``on the Farey side'' and ``on the barycentric side'', respectively. Then, at the $n$th stage, $\angles{v^i_1,v^i_2,v^i_3}$ is subdivided into three subtriangles $\angles{v^i_1,v^i_2,w^i}$, $\angles{v^i_1,w^i,v^i_3}$, $\angles{w^i,v^i_2,v^i_3}$, where $w^1$ is the Farey sum of $v^1_1,v^1_2,v^1_3$, and $w^2$ is the barycentric sum of $v^2_1,v^2_2,v^2_3$. The new triangles are paired in the obvious way, and the function $\delta$ is defined using an appropriate limiting process.
This $\delta$ function is not injective, nor is continuous at all points~\cite[p.~117]{beavergar04}. This is essentially due to the fact that not every sequence of nested Farey triangles
$\angles{v^1_1(0),v^1_2(0),v^1_3(0)}\supset
\angles{v^1_1(1),v^1_2(1),v^1_3(1)}\supset
\angles{v^1_1(2),v^1_2(2),v^1_3(2)}\supset\cdots$
intersects in a single point (here, for every $n$, $\angles{v^1_1(n),v^1_2(n),v^1_3(n)}$ is one of the three triangles resulting from the subdivision of $\angles{v^1_1(n-1),v^1_2(n-1),v^1_3(n-1)}$
at stage $n$). In terms of multidimensional continued fractions~\cite{brentjes81} \cite{schweiger00}, this amounts to saying that the continued fraction algorithm naturally associated with the $2$-dimensional Farey partition is not topologically convergent~\cite[Definition~9]{schweiger00}.

In this paper we will construct another generalization of the Minkowski function, by replacing the $2$-dimensional Farey continued fraction algorithm with the $n$-dimensional \monk{} algorithm. The latter algorithm is topologically convergent, and this fact allows us to construct, for every $n\ge2$, an $n$-dimensional Minkowski function $\Phi$ which is a homeomorphism. We will show that appropriate analogs of the properties (1)--(5) continue to hold, with the exception of (2), for which we have partial results.

A remark on terminology is in order here: the multidimensional continued fraction algorithm we are going to use has been rediscovered over and over again. We call it the \newword{\monk\ algorithm} ---and we call \newword{\monk\ map} the associated piecewise-fractional map--- since the first reference we are aware of is~\cite{monkemeyer54}. 
The name \newword{Selmer algorithm} is more widely used; as a matter of fact, the \monk\ algorithm is just the restriction of the Selmer one~\cite{selmer61} to the absorbing simplex~$D$ of~\cite[Theorem~22]{schweiger00}.
In~\cite{baldwin92a} the same algorithm is called the GMA (generalized mediant algorithm), and is defined on a simplex obtainable from $D$ via a permutation of the coordinates. All these versions are easily shown to be equivalent to each other.

\section{An $n$-dimensional Minkowski function}

We will define our generalization $\Phi$ of the Minkowski function as the only homeomorphism of a certain $n$-dimensional simplex $\Delta$ that conjugates the \monk{} map $M$ with a version of the tent map $T$, both maps to be defined shortly. In order to streamline the presentation, we fix some notation. First of all, we fix an integer $n\ge1$, and we identify $\Rbb^n$ with the plane $\pi=\{x_{n+1}=1\}$ in $\Rbb^{n+1}$. If $\pp{v}=(\alpha_1,\ldots,\alpha_n,\alpha_{n+1})\in\Rbb^{n+1}$
and $\alpha_{n+1}>0$, we denote the projection of $\pp{v}$ on $\pi$ by $v=(\alpha_1/\alpha_{n+1},\ldots,\alpha_n/\alpha_{n+1})$. Conversely, if $v\in\Qbb^n$, then we denote by $\pp{v}$ the unique point $\pp{v}=(l_1,\ldots,l_n,l_{n+1})\in\Zbb^{n+1}$ such that $\vect l1{{n+1}}$ are relatively prime, $l_{n+1}>0$, and $\pp{v}$ projects to $v$. In this case, we say that $v$ is a \newword{rational point} and that
the coordinates of $\pp{v}$ are the \newword{primitive projective coordinates} of $v$. Note that this convention differs from the one used in~\cite{schweiger00}, where projective coordinates range from $0$ to $n$, and $\pi=\{x_0=1\}$.

An $n$-dimensional simplex in $\Rbb^n$ is \newword{unimodular} if its vertices $\vect v1{{n+1}}$ are rational and $\vect{\pp{v}}1{{n+1}}$ constitute a $\Zbb$-basis of $\Zbb^{n+1}$. In all this paper, $\Delta$ will denote the simplex whose vertices $\vect v1{{n+1}}$ are given, in primitive projective coordinates, by the columns of the following matrix:
$$
V=\begin{pmatrix}
0 & 1 & 1 & \cdots & 1 & 1 \\
0 & 1 & 1 & \cdots & 1 & 0 \\
0 & 1 & 1 & \cdots & 0 & 0 \\
\vdots & \vdots & \vdots & & \vdots & \vdots \\
0 & 1 & 0 & \cdots & 0 & 0 \\
1 & 1 & 1 & \cdots & 1 & 1
\end{pmatrix}.
$$
More precisely, all entries of $V$ are $0$, except those in position $ij$, with either ($i=n+1$) or ($j\ge2$ and $i+j\le n+2$), that have value $1$. Clearly $\Delta$ is unimodular.

Consider now the following $(n+1)\times(n+1)$ matrices:
$$
A_0=\begin{pmatrix}
1 & 0 & 0 & \cdots & 0 & 1 \\
0 & 0 & 0 & \cdots & 0 & 1 \\
0 & 1 & 0 & \cdots & 0 & 0 \\
\vdots & \vdots & \vdots & & \vdots & \vdots \\
0 & 0 & 0 & \cdots & 0 & 0 \\
0 & 0 & 0 & \cdots & 1 & 0
\end{pmatrix},
\qquad
A_1=\begin{pmatrix}
0 & 0 & 0 & \cdots & 0 & 1 \\
1 & 0 & 0 & \cdots & 0 & 1 \\
0 & 1 & 0 & \cdots & 0 & 0 \\
\vdots & \vdots & \vdots & & \vdots & \vdots \\
0 & 0 & 0 & \cdots & 0 & 0 \\
0 & 0 & 0 & \cdots & 1 & 0
\end{pmatrix}.
$$
Here, all entries of $A_0$ and $A_1$ are $0$, except $(A_0)_{11}$, $(A_1)_{21}$, and all elements in position $1(n+1)$, $2(n+1)$, $(j+1)j$, for $2\le j\le n$, that have value $1$.
Let $\Delta_0,\Delta_1$ be the unimodular simplexes whose vertices are given, in projective coordinates, by the columns of $VA_0$ and $VA_1$, respectively. For $a=0,1$, the matrix $M_a=VA_a^{-1}V^{-1}\in\GL_{n+1}(\Zbb)$ expresses, in projective coordinates, a fractional-linear homeomorphism ---with a slight abuse of notation, again denoted by $M_a$--- from $\Delta_a$ to $\Delta$ as follows: if $x=(x_1\cdots x_n)^{tr}\in\Delta_a$, then the projective coordinates of $M_a(x)$ are $M_a(x_1\cdots x_n\,1)^{tr}$. 
Note that $\Delta=\Delta_0\cup\Delta_1$ and $M_0=M_1$ on $\Delta_0\cap\Delta_1$. Indeed, the $(n-1)$-simplex $\Delta_0\cap\Delta_1$ has vertices given by the columns of $VA'$ (where $A'$ is the $(n+1)\times n$ matrix obtained from either $A_0$ or $A_1$ by removing the first column) and $M_0VA'=M_1VA'$.
We remark here, for future reference in the proof of Proposition~\ref{ref4}, that $M_0[\Delta_0\cap\Delta_1]=M_1[\Delta_0\cap\Delta_1]$ is the $(n-1)$-dimensional face of $\Delta$ whose vertices are $\vect v2{{n+1}}$ (just consider the columns of $VA_0^{-1}A'$).
The continuous piecewise-fractional map $M:\Delta\to\Delta$ defined by $M(x)=M_a(x)$, for $x\in\Delta_a$, is the \newword{\monk{} map}.
A simple matrix computation shows that $\Delta_0=\{x\in\Delta:x_1+x_n\le1\}$ and that, in affine coordinates,
$$
M(x_1,\vect x2n)=\begin{cases}
\displaystyle{\Big(\frac{x_1}{1-x_n},\frac{x_1-x_n}{1-x_n},\ldots,\frac{x_{n-1}-x_n}{1-x_n}\Big)}, & \text{if $x_1+x_n\le1$;} \\
\\
\displaystyle{\Big(\frac{1-x_n}{x_1},\frac{x_1-x_n}{x_1},\ldots,\frac{x_{n-1}-x_n}{x_1}\Big)}, & \text{if $x_1+x_n\ge1$.}
\end{cases}
$$

For $a=0,1$, let now $B_a$ be the $(n+1)\times(n+1)$ matrix which is identical to $A_a$ except for the last column, where the two $1$'s are replaced by $1/2$. The matrices $V$ and $VB_a$ agree in the last row $(1\cdots 1\,1)$. Therefore, the product of the first one with the inverse of the second, i.e., the matrix $T_a=VB_a^{-1}V^{-1}$, has last row $(0\cdots0\,1)$ and determines an affine map $T_a:\Delta_a\to\Delta$ as follows: if $x=(x_1\cdots x_n)^{tr}\in\Delta_a$ and $y=(y_1\cdots y_n)^{tr}=T_a(x)$, then $T_a(x_1\cdots x_n\,1)^{tr}=(y_1\cdots y_n\,1)^{tr}$.
As above, $T_0=T_1$ on $\Delta_0\cap\Delta_1$. 
The continuous piecewise-affine map $T:\Delta\to\Delta$ defined by $T(x)=T_a(x)$, for $x\in\Delta_a$, is the \newword{tent map}.
In affine coordinates, $T$ is expressed by
$$
T(x_1,\vect x2n)=\begin{cases}
(x_1+x_n,x_1-x_n,\ldots,x_{n-1}-x_n), & \text{if $x_1+x_n\le1$;} \\
(2-x_1-x_n,x_1-x_n,\ldots,x_{n-1}-x_n), & \text{if $x_1+x_n\ge1$.}
\end{cases}
$$
Note that, for $n=1$, the \monk{} map and the tent map defined above coincide with the Farey map and the tent map of \S\ref{ref2}.
The following theorem is then an $n$-dimensional generalization of Proposition~\ref{ref3}.

\begin{theorem}
There\label{ref1} exists a unique homeomorphism $\Phi:\Delta\to\Delta$ such that $T=\Phi\circ M\circ\Phi^{-1}$.
\end{theorem}

The rest of this section is devoted to the proof of Theorem~\ref{ref1}; we first prove the existence of $\Phi$, then its uniqueness.
Recall that a \newword{[rational] simplicial complex} in $\Rbb^n$ is a finite set $\Ccal$ of simplexes in $\Rbb^n$ such that:
(1) all vertices of all simplexes in $\Ccal$ are rational; (2) if $\Gamma\in\Ccal$ and $\Sigma$ is a face of $\Gamma$, then $\Sigma\in\Ccal$; (3) every two simplexes intersect in a common face.
The \newword{support} of $\Ccal$ is the set-theoretic union $\abs{\Ccal}$ of all simplexes in $\Ccal$. A complex $\Ccal$ \newword{refines} a complex $\Dcal$, written $\Ccal\ge\Dcal$, if $\abs{\Ccal}=\abs{\Dcal}$ and every simplex of $\Ccal$ is contained in some simplex of $\Dcal$.
The \newword{mesh} of $\Ccal$, written $\mesh(\Ccal)$ is the maximum diameter of the simplexes in $\Ccal$ or, equivalently~\cite[Corollary~5.18]{hockingyou}, the maximum
length of the $1$-simplexes in $\Ccal$.

The set $\Fcal_1$ of all faces of $\Delta_0$ and $\Delta_1$ is a simplicial complex supported in $\Delta$. For short, we list only the maximal (w.r.t.~the relation of being a face) simplexes, thus writing $\Fcal_1=\{\Delta_0,\Delta_1\}$; we also write $\Fcal_0=\{\Delta\}$.
For every finite sequence $\vect a0{{t-1}}\in\{0,1\}$, we define by recursion
\begin{align*}
\Delta_{a_0\ldots a_{t-1}} &= \Delta_{a_0}\cap M^{-1}\Delta_{a_1\ldots a_{t-1}}\\
&=\{x:x\in\Delta_{a_0} \lland M(x)\in\Delta_{a_1} \lland M^2(x)\in\Delta_{a_2}
\lland \cdots \lland M^{t-1}(x)\in\Delta_{a_{t-1}}\},
\end{align*}
and we call $\Fcal_t=\{\Delta_{a_0\ldots a_{t-1}}:\vect a0{{t-1}}\in\{0,1\}\}$ the \newword{time-$t$ partition} for $M$.

\begin{proposition}
Every\label{ref4} $\Fcal_t$ is a simplicial complex, whose maximal elements are the $2^t$ $n$-simplexes $\Delta_{a_0\ldots a_{t-1}}$. For every $t\ge0$, the complex $\Fcal_{t+1}$ refines $\Fcal_t$.
\end{proposition}
\begin{proof}
Let $\psi_a=M_a^{-1}:\Delta\to\Delta_a$ be the two inverse branches of $M$, for $a=0,1$. Note that $\Delta_{a_0\ldots a_{t-1}}=\psi_{a_0}\circ\psi_{a_1}\circ\cdots\circ\psi_{a_{t-1}}[\Delta]$. Indeed, this is true for $t=1$, and follows by induction otherwise, since
\begin{align*}
\Delta_{a_0\ldots a_{t-1}} &= \Delta_{a_0}\cap M^{-1}\Delta_{a_1\ldots a_{t-1}}\\
&= \Delta_{a_0}\cap\bigl[\psi_0[\Delta_{a_1\ldots a_{t-1}}]\cup
\psi_1[\Delta_{a_1\ldots a_{t-1}}]\bigr]\\
&= \psi_{a_0}[\Delta_{a_1\ldots a_{t-1}}].
\end{align*}
We now proceed by induction: $\Fcal_0$ and $\Fcal_1$ are simplicial complexes supported on $\Delta$, with $\Fcal_1\ge\Fcal_0$. Assuming that $\Fcal_t$ is such a complex, then the elements of $\Fcal_{t+1}$ are given by $\psi_0[\Fcal_t]\cup\psi_1[\Fcal_t]$, where $\psi_a[\Fcal_t]$ is the set of all $\psi_a$-images of the elements of $\Fcal_t$. Since $\psi_a$ is fractional-linear, $\psi_a[\Fcal_t]$ is a simplicial complex supported in $\Delta_a$. It is therefore sufficient to show that $\psi_0[\Fcal_t]$ and $\psi_1[\Fcal_t]$ agree (i.e., induce the same complex) on the intersection $\Delta_0\cap\Delta_1$. This fact is true 
because, as we remarked in the course of the definition of the \monk{} map, $M_0$ and $M_1$ agree on $\Delta_0\cap\Delta_1$, and provide a fractional-linear homeomorphism between $\Delta_0\cap\Delta_1$ and the $(n-1)$-dimensional face $\Lambda$ of $\Delta$ whose vertices are $\vect v2{{n+1}}$. Therefore $\psi_0$ and $\psi_1$ agree on $\Lambda$. This implies that $\psi_0[\Fcal_t]$ and $\psi_1[\Fcal_t]$ 
induce the same complex on $\Delta_0\cap\Delta_1$, namely the $\psi_0$-image, which is also the $\psi_1$-image, of the complex induced by $\Fcal_t$ on $\Lambda$. The fact that $\Fcal_{t+1}$ refines $\Fcal_t$ is immediate, since every maximal simplex $\Delta_{a_0\ldots a_{t-1}a_t}$ is contained in $\Delta_{a_0\ldots a_{t-1}}$.
\end{proof}

We construct analogously the time-$t$ partition $\Bcal_t$ for $T$. Namely, we let $\Bcal_0=\Fcal_0$, $\Gamma_0=\Delta_0$, $\Gamma_1=\Delta_1$, and $\Gamma_{a_0\ldots a_{t-1}}=\Gamma_{a_0}\cap T^{-1}\Gamma_{a_1\ldots a_{t-1}}$. The analogue of Proposition~\ref{ref4} holds verbatim, and we have simplicial complexes $\Bcal_t=\{\Gamma_{a_0\ldots a_{t-1}}:\vect a0{{t-1}}\in\{0,1\}\}$, with $\Bcal_{t+1}$ refining $\Bcal_t$.
An obvious induction on $t$ shows that there exists a unique combinatorial isomorphism from $\Fcal_t$ to $\Bcal_t$ that fixes the vertices of $\Delta$. At the level of maximal simplexes, the isomorphism is given by $\Delta_{a_0\ldots a_{t-1}}\mapsto
\Gamma_{a_0\ldots a_{t-1}}$.
We draw a picture of $\Fcal_4$ and $\Bcal_4$, for $n=2$, labeling the $2$-simplex $\Gamma_{a_0a_1a_2a_3}\in\Bcal_4$ by $a_0a_1a_2a_3$.
\begin{figure}[H]
\begin{center}
\includegraphics[height=4.5cm,width=4.5cm]{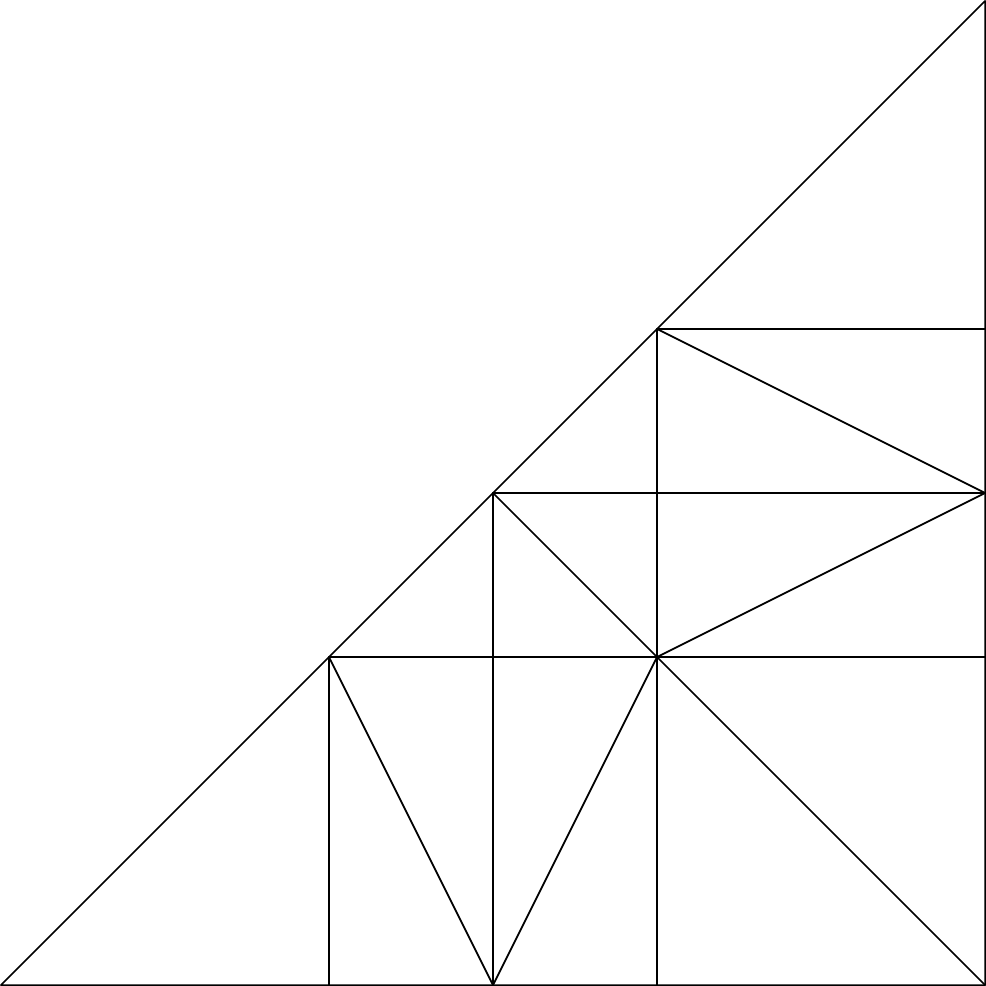}
\qquad
\includegraphics[height=4.5cm,width=4.5cm]{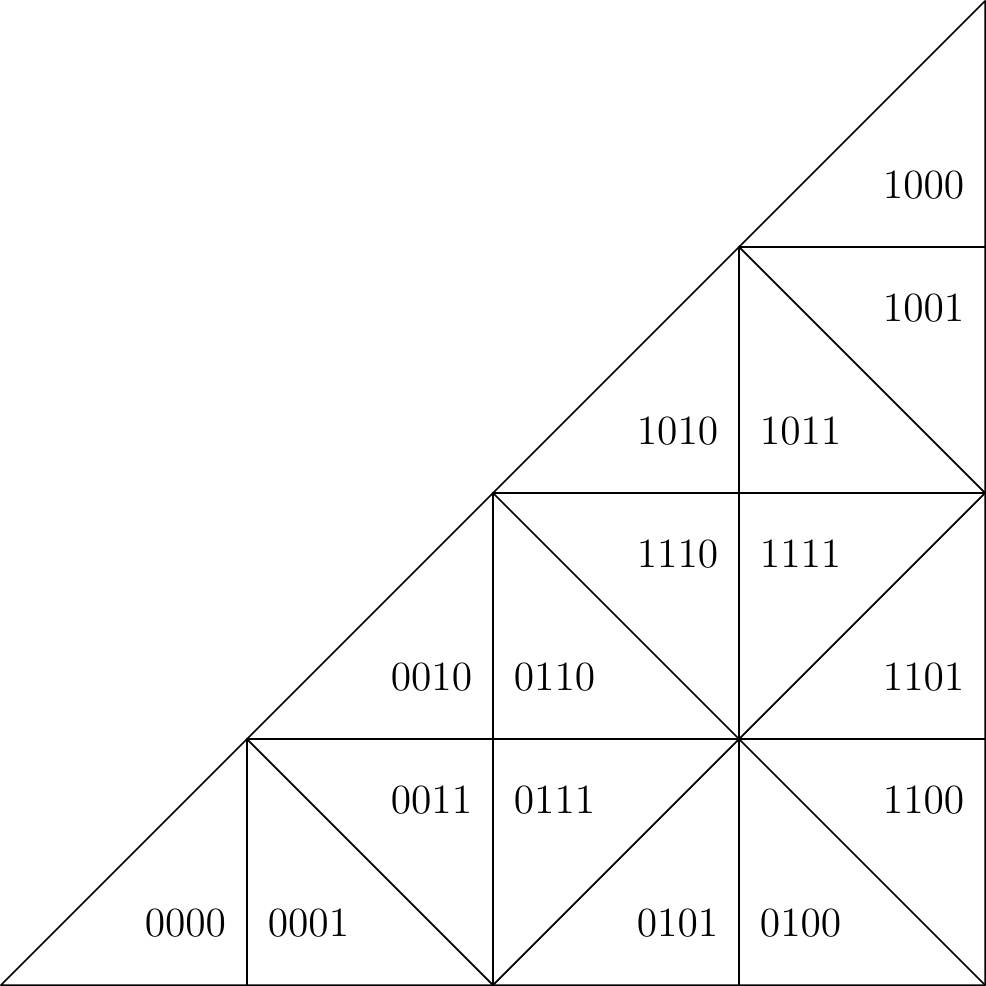}
\end{center}
\end{figure}

Let $\cantor$ be the \newword{Cantor space}, i.e., the set of all infinite sequences $\abar=a_0a_1a_2\ldots$ of elements of $\{0,1\}$, endowed with the product topology. For $t\ge1$, we write $\abar\restriction t$ for $a_0a_1\ldots a_{t-1}$, and we let $[a_0\ldots a_{t-1}]$ be the
\newword{cylinder} $\{\bbar:\abar\restriction t=\bbar\restriction t\}$; we extend this convention by setting $\Delta_{\abar\restriction0}=\Gamma_{\abar\restriction0}=\Delta$ and $[a_0\ldots a_{-1}]=\cantor$.

\begin{lemma}
For\label{ref7} every $\abar\in\cantor$, the intersection $\bigcap_{t\ge0}\Delta_{\abar\restriction t}$ is a singleton, and the intersection $\bigcap_{t\ge0}\Gamma_{\abar\restriction t}$ is a singleton as well.
\end{lemma}
\begin{proof}
The first statement amounts to saying that the \monk{} algorithm is topologically convergent~\cite[Definition~9]{schweiger00}: this fact is proved in~\cite[Satz~10]{monkemeyer54} as well as in~\cite[Lemma~19]{schweiger00}. Note that~\cite[Satz~10]{monkemeyer54} assumes that $\bigcap_{t\ge0}\Delta_{\abar\restriction t}$ contains a point whose coordinates are not all rational. But this is not a restriction since, if $\bigcap_{t\ge0}\Delta_{\abar\restriction t}$ contained two distinct points $p,q$ then, by convexity, it would contain all the points in the line segment connecting $p$ with $q$, and hence a point whose coordinates are not all rational.

In order to prove the second statement note that the vertices of $\Gamma_{\abar\restriction t}$ are given, in projective coordinates, by the columns of $VB_{a_0}\cdots B_{a_{t-1}}$. Let $\stochastic$ be the set of all $(n+1)\times(n+1)$ column-stochastic matrices (i.e., all nonnegative matrices having the property that the entries in each column sum up to $1$). Observe that $\stochastic$ is a compact submonoid of $(\Mat_{(n+1)\times(n+1)}\Rbb,\cdot,\Id)$. Let $B_{\abar\restriction t}=B_{a_0}\cdots B_{a_{t-1}}\in\stochastic$. 
We will apply~\cite[Theorem~6.1]{daubechieslag} to show that, for every $\abar\in\cantor$, the limit $\Btilde=\lim_{t\to\infty}B_{\abar\restriction t}$ exists (necessarily in $\stochastic$, since the latter is closed), and all columns of $\Btilde$ are equal. 
Recall that a column-stochastic matrix $C=C_{ij}$ is \newword{$(j_1,j_2)$-scrambling} if there exists a row index $i$ such that
$C_{ij_1}$ and $C_{ij_2}$ are both $>0$; $C$ is \newword{scrambling} if it is $(j_1,j_2)$-scrambling for every pair $(j_1,j_2)$ of columns indices~\cite{hajnal58}. By~\cite[Theorem~6.1]{daubechieslag}, it will be sufficient to prove the following:
\begin{itemize}
\item[(A)] there exists $s>0$ such that all products of $s$ matrices from $\{B_0,B_1\}$ (repetitions allowed) are scrambling.
\end{itemize}
It is simpler to argue on the incidence graphs $G(B_0)$ and $G(B_1)$ associated to $B_0$ and $B_1$. The graph $G(B_0)$ has $n+1$ vertices and there is a directed edge connecting the $j$th vertex to the $i$th iff $(B_0)_{ij}>0$; similarly for $G(B_1)$. We combine $G(B_0)$ and $G(B_1)$ in a single graph $G$ as in the following picture, with the understanding that in $G(B_0)$ the edge $e_0$ is activated and the edge $e_1$ is discarded, and conversely in $G(B_1)$.
\begin{figure}[H]
\begin{center}
\includegraphics[width=7cm]{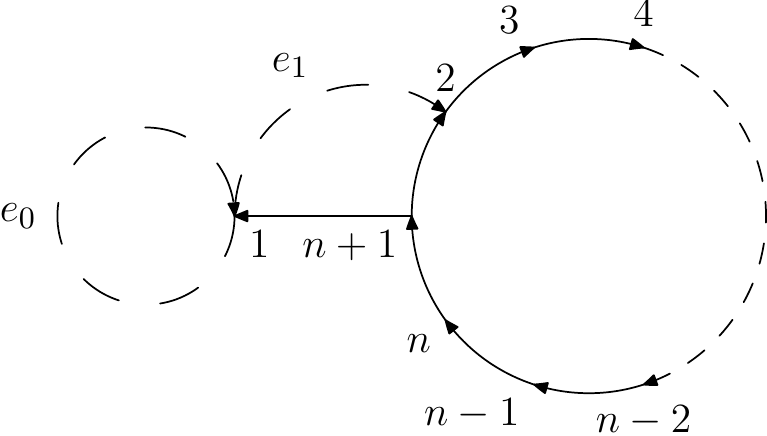}
\end{center}
\end{figure}
We will deduce property (A) from the existence of a winning strategy for a certain game on $G$. The game starts with two Lovers sitting in distinct vertices. A move of the game consists of the following: first, the Enemy chooses which of the two edges $e_0$ and $e_1$ is to be active at that move, and then each Lover moves one step along an edge departing from his/her current vertex. The Lovers win the game if after finitely many moves they end up in the same vertex.

\noindent\emph{Claim.} Regardless of the initial position, the Lovers win in at most $(n+1)n/2$ moves.

Assuming the Claim, let us prove (A). We take $s=(n+1)n/2$, and we fix a product $B=B_{a_0}\cdots B_{a_{s-1}}$ of $s$ matrices from $\{B_0,B_1\}$. No column in $B_0$ or in $B_1$ is identically $0$; therefore, for every $(j_1,j_2)$-scrambling matrix $C$, both $B_0C$ and $B_1C$ are $(j_1,j_2)$-scrambling. Choose column indices $j_1,j_2$; by the above, we can assume $j_1\not=j_2$. Consider the game in which the Lovers apply the winning strategy and start in position $j_1$ and $j_2$, while the Enemy activates the edge $e_{a_{s-r}}$ at the $r$th move ($r\ge1$). By the Claim, this game ends after $1\le t\le s$ moves, leaving the Lovers in the same vertex $i$. This implies that there exists a path in $G$ connecting $j_1$ to $i$, and such that the $r$th edge in the path is an edge of $G(B_{a_{s-r}})$. By the elementary properties of the incidence graphs of nonnegative square matrices, 
the $ij_1$th entry of $B_{a_{s-t}}\cdots B_{a_{s-2}}B_{a_{s-1}}$ is $>0$. Analogously, the $ij_2$th entry is $>0$; hence $B_{a_{s-t}}\cdots B_{a_{s-2}}B_{a_{s-1}}$ is $(j_1,j_2)$-scrambling, and so is $B$. Since $j_1$ and $j_2$ are arbitrary, $B$ is scrambling.

\noindent\emph{Proof of Claim.} Given any vertex $p$ of $G$, there exists a unique vertex path $p=p_0,p_1,p_2,\ldots$ such that, for every $i\ge1$, there is an oriented edge in $G$, different from both $e_0$ and $e_1$, which connects $p_i$ to $p_{i-1}$. Let us call such a path a \newword{backward path}. The length of a finite backward path 
$p_0,p_1,p_2,\ldots,p_r$ is $r$, its origin is $p_0$, and its endpoint is $p_r$. Let $\Vcal$ be the set, of cardinality $(n+1)n/2$, whose elements are all unordered pairs of distinct vertices of $G$. If $\{p,q\}\in\Vcal$, then the \newword{gap} $g(p,q)$ of $\{p,q\}$ is the minimal length of a backward path whose origin is one of the vertices $p,q$, and whose endpoint is the other vertex. The origin of such a path is the \newword{leading vertex} of $\{p,q\}$, and the \newword{defect} $d(p,q)$ of $\{p,q\}$ is $0$ if the leading vertex is $1$, and is $g(1,\text{leading vertex})$ otherwise. The leading vertex, and the numbers $1\le g(p,q)\le n$ and $0\le d(p,q)\le n$, are uniquely determined by the pair, with the exception of the case in which $n$ is even, $p$ and $q$ are both $\not=1$, and $g(p,q)=n/2$. In this case, we define the leading vertex to be the vertex whose gap from $1$ is minimal, and we define the defect accordingly. Let $\Tcal$ be the set $\{1,\ldots,n\}\times\{0,\ldots,n\}$, ordered lexicographically: $(g,d)\prec(g',d')$ iff $g<g'$ or ($g=g'$ and $d<d'$). The map $\chi:\Vcal\to\Tcal$ defined by $\chi\{p,q\}=\bigl(g(p,q),d(p,q)\bigr)$ is injective: indeed $d(p,q)$ determines uniquely the leading vertex of the pair (start from $1$ and go backwards $d(p,q)$ steps in $G$, never using the edges $e_0$ and $e_1$), and then $g(p,q)$ determines the other vertex.

Assume now that at a certain stage of the game the Lovers are in distinct vertices $p,q$, and consider the following strategy.
\begin{itemize}
\item[(a)] If $\{p,q\}=\{1,n+1\}$, then the Lovers win at the next move, either by meeting at $1$ (if the Enemy chooses to activate $e_0$) or by meeting at $2$ (if the Enemy activates $e_1$).
\item[(b)] Otherwise, if $n+1\notin\{p,q\}$, then each Lover follows the unique edge at his disposal.
\item[(b)] Otherwise, without loss of generality $p=n+1$ and $q\notin\{1,n+1\}$. Then the Lover at $q$ follows the unique edge at his disposal, while the Lover at $p$ moves to $1$ provided that $p$ is the leading vertex of the pair $\{p,q\}$, otherwise moves to $2$.
\end{itemize}
Let $p',q'$ be the vertices occupied by the Lovers at the next step, and assume that the game is not finished yet (hence case~(a) did not apply, and $p'\not=q'$).
An easy case analysis, distinguishing the three cases (i) the leading vertex is $1$, (ii) the leading vertex is $n+1$, and (iii) the leading vertex is in $\{2,\ldots,n\}$, shows that $\chi(p',q')\prec\chi(p,q)$. In other words, at each step either the gap of the pair decreases, or the gap stays the same and the defect decreases. Since $\chi[\Vcal]$ has finite size and is totally ordered by $\prec$, the length of the longest possible game, final winning move included, coincides with this size, namely $(n+1)n/2$.
\end{proof}

Note that the bound in the proof of Lemma~\ref{ref7} is sharp: for $n=2$ the matrix $A_0A_1$ is not scrambling, and for $n=3$ the matrix $A_0A_0A_0A_1A_0$ is not scrambling either.

\begin{corollary}
Both\label{ref9} $\lim_{t\to\infty}\mesh(\Fcal_t)$ and 
$\lim_{t\to\infty}\mesh(\Bcal_t)$ exist and have value~$0$.
\end{corollary}
\begin{proof}
Suppose that statement is false for, say, the Farey complexes. Then there exists $\varepsilon>0$ such that, for every $t$, the set
$$
\Scal_t=\{a_0a_1\ldots a_{t-1}:\text{the diameter of $\Delta_{a_0\ldots a_{t-1}}$ is $\ge\varepsilon$}\}
$$
is not empty. If $a_0a_1\ldots a_{t-1}\in\Scal_t$ and $0<k\le t$, then $a_0a_1\ldots a_{k-1}\in\Scal_k$, so the union of all $\Scal_t$'s is an infinite subtree of the full binary tree. By K\"onig's Lemma~\cite[Lemma~3.3.19]{vandalen94} this subtree has an infinite branch $\abar\in\cantor$. This contradicts Lemma~\ref{ref7}.
\end{proof}

As a side remark note that, for $n=2$, all the $2^t$ triangles in $\Bcal_t$ are congruent, because the $2\times 2$ matrices obtained from $T_0^{-1}$ and $T_1^{-1}$ by removing the third row and the third column are both of the form $1/\sqrt{2}\cdot(\text{an orthogonal matrix})$. This is no longer true for $n\ge3$.

Given $\abar\in\cantor$, let $\varphi(\abar)$ be the unique point in $\bigcap_{t\ge0}\Delta_{\abar\restriction t}$, and let $\upsilon(\abar)$ be the unique point in 
$\bigcap_{t\ge0}\Gamma_{\abar\restriction t}$.

\begin{lemma}
The\label{ref6} mappings $\varphi,\upsilon:\cantor\to\Delta$ are continuous, surjective, and have the same fibers (i.e., $\varphi(\abar)=\varphi(\bbar)$ iff $\upsilon(\abar)=\upsilon(\bbar)$).
\end{lemma}
\begin{proof}
Clearly $\varphi$ is surjective and we have
\begin{align*}
\Delta_{a_0\ldots a_{t-1}}&\subseteq\varphi\bigl[[a_0\ldots a_{t-1}]\bigr];\tag{$*$} \\
[a_0\ldots a_{t-1}]&\subseteq\varphi^{-1}\Delta_{a_0\ldots a_{t-1}}\tag{$**$}.
\end{align*}
If $U\subseteq\Delta$ is open, then we have
\begin{equation*}
\varphi^{-1}U=\bigcup\{[a_0\ldots a_{t-1}]:\Delta_{a_0\ldots a_{t-1}}\subseteq U\}\tag{$*{*}*$}.
\end{equation*}
Indeed, the $\supseteq$ inclusion is immediate from~($**$). On the other hand, let $\varphi(\abar)\in U$. Then $(\Delta\setminus U)\cap\bigcap_{t\ge0}\Delta_{\abar\restriction t}=\emptyset$, and hence by compactness there exists $t\ge0$ such that $\Delta_{\abar\restriction t}\subseteq U$. Therefore $\abar$ belongs to the right-hand side of~($*{*}*$), and equality follows. Since the right-hand side of~($*{*}*$) is open in the Cantor space, $\varphi$ is continuous.
Exactly the same proof shows that $\upsilon$ is surjective and continuous as well.

We assume now that $\abar,\bbar$ are such that $\varphi(\abar)=\varphi(\bbar)$, and prove $\upsilon(\abar)=\upsilon(\bbar)$. By hypothesis, for every $t\ge0$ we have $\Delta_{\abar\restriction t}\cap\Delta_{\bbar\restriction t}\not=\emptyset$, and hence $\Delta_{\abar\restriction t}$ and $\Delta_{\bbar\restriction t}$ intersect in a common nonempty face. As observed above, $\Fcal_t$ and $\Bcal_t$ are combinatorially isomorphic; therefore, for every $t$, $\Gamma_{\abar\restriction t}$ and $\Gamma_{\bbar\restriction t}$ intersect in a common nonempty face as well. Again by compactness, $\bigcap_{t\ge0}(\Gamma_{\abar\restriction t}\cap\Gamma_{\bbar\restriction t})\not=\emptyset$. Since by definition $\bigcap_{t\ge0}\Gamma_{\abar\restriction t}=\{\upsilon(\abar)\}$ and
$\bigcap_{t\ge0}\Gamma_{\bbar\restriction t}=\{\upsilon(\bbar)\}$,
we have $\upsilon(\abar)=\upsilon(\bbar)$. Clearly the r\^ole of $\varphi$ and $\upsilon$ can be reversed, and it follows that $\varphi$ and $\upsilon$ have the same fibers.
\end{proof}

Let $\equiv$ be the equivalence relation on the Cantor space defined by $\abar\equiv\bbar$ iff $\varphi(\abar)=\varphi(\bbar)$ iff $\upsilon(\abar)=\upsilon(\bbar)$. Let $\chi:\cantor\to\cantor/\equiv$ be the quotient mapping, and endow $\cantor/\equiv$ with the quotient topology: $V$ is open in $\cantor/\equiv$ iff $\chi^{-1}V$ is open in $\cantor$. We have an obvious factorization in continuous mappings
$$
\begin{xy}
\xymatrix{
& \cantor \ar[dl]_\varphi \ar[d]^\chi \ar[dr]^\upsilon  & \\
\Delta & \cantor/\equiv \ar[l]^{\bar\varphi} \ar[r]_{\bar\upsilon} & \Delta
}
\end{xy}
$$
The quotient space $\cantor/\equiv$ is compact, and $\Delta$ is Hausdorff. Hence the continuous bijections $\bar\varphi$ and $\bar\upsilon$ are both homeomorphisms.

\begin{definition}
We define $\Phi:\Delta\to\Delta$ as the homeomorphism $\Phi={\bar\upsilon}\circ{\bar\varphi}^{-1}$. Equivalently, $\Phi(p)=\upsilon(\abar)$, for any $\abar$ such that $\varphi(\abar)=p$.
\end{definition}

For every $\vect a0{{t-1}}$, the homeomorphism $\Phi$ restricts to a bijection from $\Delta_{a_0\ldots a_{t-1}}$ to $\Gamma_{a_0\ldots a_{t-1}}$; this follows from~($*$) and~($**$) in the proof of Lemma~\ref{ref6} and the corresponding inclusions for~$\upsilon$.
For every $t>1$ we have $M[\Delta_{a_0\ldots a_{t-1}}]=M\circ\psi_{a_0}
[\Delta_{a_1\ldots a_{t-1}}]=\Delta_{a_1\ldots a_{t-1}}$; this makes sense also for $t=1$, since $M[\Delta_{a_0}]=\Delta$. Analogously we have $T[\Gamma_{a_0\ldots a_{t-1}}]=\Gamma_{a_1\ldots a_{t-1}}$. Denoting by
$S$ the shift map on $\cantor$ (i.e., $S(a_0a_1a_2\ldots)=(a_1a_2a_3\ldots)$), it easily follows that the following diagram commutes:
$$
\begin{xy}
\xymatrix{
\Delta \ar@/_2pc/[dd]_\Phi \ar[r]^M  & \Delta \ar@/^2pc/[dd]^\Phi \\
\cantor \ar[u]^\varphi \ar[d]_\upsilon \ar[r]^S & \cantor \ar[u]_\varphi \ar[d]^\upsilon \\
\Delta \ar[r]_T & \Delta
}
\end{xy}
$$
By chasing the diagram, one sees immediately that $T=\Phi\circ M\circ\Phi^{-1}$, as required.

We now show the uniqueness of $\Phi$, by assuming that $\Psi$ is a homeomorphism of $\Delta$ such that $T=\Psi\circ M\circ\Psi^{-1}$ and proving that $\Psi=\Phi$. Observe that the boundary $\partial\Delta$ of $\Delta$ is characterized ---in purely topological terms, with no reference to the immersion of $\Delta$ in $\Rbb^n$--- as the set of points $p\in\Delta$ whose removal leaves $\Delta\setminus\{p\}$ contractible. Therefore, $\Phi[\partial\Delta]=\Psi[\partial\Delta]=\partial\Delta$. Let $\Sigma$ be the set-theoretic union of the proper faces of $\Delta_0$ and $\Delta_1$; we have $\Sigma=M^{-1}\partial\Delta=T^{-1}\partial\Delta$ and, as a consequence, $\Sigma=\Phi^{-1}\Sigma=\Psi^{-1}\Sigma$. Indeed, e.g., $p\in\Sigma$ iff $M(p)\in\partial\Delta$ iff $\Psi\circ M(p)\in\partial\Delta$ iff
$T\circ\Psi(p)\in\partial\Delta$ iff $\Psi(p)\in T^{-1}\partial\Delta=\Sigma$ iff $p\in\Psi^{-1}\Sigma$.

Note that $\Phi(v_1)=\Psi(v_1)=v_1$. Indeed, $v_1$ is a point in $\partial\Delta$ which is fixed by $M$, and therefore both $\Phi(v_1)$ and $\Psi(v_1)$ must be points in $\partial\Delta$ which are fixed by~$T$. But there is only one such point, namely $v_1$ itself. Observe now that $\Delta_0$ can be characterized as the set of points in $\Delta$ that can be connected to $v_1$ by a continuous path whose relative interior does not intersect $\Sigma$. Since $\Phi$ and $\Psi$ are homeomorphisms, both fixing $\Sigma$ globally, we can safely conclude that $\Phi[\Delta_a]=\Psi[\Delta_a]$, for $a=0,1$.

Let now $p$ be any point of $\Delta$, and choose $\abar$ such that $\varphi(\abar)=p$ and $\upsilon(\abar)=\Phi(p)$. For every $t\ge0$ we have $M^t(p)\in\Delta_{a_t}=\Gamma_{a_t}$, and therefore $\Psi^{-1}\circ T^t\bigl(\Psi(p)\bigr)\in\Gamma_{a_t}$, i.e., $T^t\bigl(\Psi(p)\bigr)\in\Psi[\Gamma_{a_t}]=\Gamma_{a_t}$. Hence, by definition of $\upsilon$, $\Psi(p)=\upsilon(\abar)=\Phi(p)$. This concludes the proof of Theorem~\ref{ref1}.

\section{Fractal structure, periodicity and singularity}

In this section we will discuss how properties \S1(1)--(5) of the classical Minkowski function generalize to our $n$-dimensional setting. Basically, all properties continue to hold, with the exception of \S1(2), whose full validity turns out to be an open problem. Let us first treat \S1(4).

\begin{proposition}
Let\label{ref10} $t\ge0$, let $\Xi\in\Fcal_t$, and let $\Lambda$ be the simplex in $\Bcal_t$ corresponding to $\Xi$ under the combinatorial isomorphism defined before Lemma~\ref{ref7}. Then $\Phi$ restricts to a homeomorphism between $\Xi$ and $\Lambda$.
Moreover, for every $\Delta_{a_0\ldots a_{t-1}}\in\Fcal_t$ we have
$$
\Phi=(T_{a_{t-1}}\circ\cdots\circ T_{a_0})\circ(\Phi\restriction\Delta_{a_0\ldots a_{t-1}})\circ(M_{a_{t-1}}\circ\cdots\circ M_{a_0})^{-1}.
$$
\end{proposition}
\begin{proof}
We can give an equivalent definition of $\Phi$ and $\Phi^{-1}$ as follows. For each $t$, we define a simplicial homeomorphism $\Phi_t:\Delta\to\Delta$ by first mapping the vertices of $\Fcal_t$ to the vertices of $\Bcal_t$ according to the combinatorial isomorphism, and then using barycentric coordinates to extend $\Phi_t$ to all of $\Delta$. More precisely, if $\Delta_{a_0\ldots a_{t-1}}$ has vertices $\vect w1{{n+1}}$ and $\Delta_{a_0\ldots a_{t-1}}\ni p=\sum\alpha_iw_i$ in barycentric coordinates, then $\Phi_t(p)=\sum\alpha_i\Phi_t(w_i)$. Using Corollary~\ref{ref9}, one sees easily that $\Phi=\lim_{t\to\infty}\Phi_t$ and 
$\Phi^{-1}=\lim_{t\to\infty}\Phi_t^{-1}$ in the topology of uniform convergence. Now, for every $p\in\Xi$ and every $k\ge t$, we have $\Phi_k(p)\in\Lambda$; since $\Lambda$ is closed, $\Phi(p)=\lim_{k\to\infty}\Phi_k(p)\in\Lambda$. Hence $\Phi[\Xi]\subseteq\Lambda$, and the same argument applied to $\Phi^{-1}$ shows that $\Phi[\Xi]=\Lambda$.

The mappings
\begin{align*}
M_{a_{t-1}}\circ\cdots\circ M_{a_0}&=M\restriction\Delta_{a_0\ldots a_{t-1}}
:\Delta_{a_0\ldots a_{t-1}}\to\Delta,\\
\intertext{and}
T_{a_{t-1}}\circ\cdots\circ T_{a_0}&=T\restriction\Gamma_{a_0\ldots a_{t-1}}
:\Gamma_{a_0\ldots a_{t-1}}\to\Delta,
\end{align*}
are both homeomorphisms, the former fractional-linear and the latter affine. Our second statement is then immediate, since it amounts to the restriction of the identity $\Phi\circ M=T\circ\Phi$ to $\Delta_{a_0\ldots a_{t-1}}$.
\end{proof}

Next, \S1(1) generalizes to the following proposition.

\begin{proposition}
$\Phi$ is an orientation-preserving homeomorphism.
\end{proposition}
\begin{proof}
Choose $t$ such that $\Fcal_t$ contains a vertex $p$ in the topological interior $\Delta^\circ$ of $\Delta$. Let $\Delta_{a_0\ldots a_{t-1}}\in\Fcal_t$, and let $D$ be the diagonal matrix whose entries along the main diagonal are those in the last row of $VA_{a_0}\cdots A_{a_{t-1}}$. Then the affine homeomorphism $\Phi_t\restriction\Delta_{a_0\ldots a_{t-1}}$ defined in the proof of Proposition~\ref{ref10} is expressed by the matrix $(VB_{a_0}\cdots B_{a_{t-1}})(VA_{a_0}\cdots A_{a_{t-1}}D^{-1})^{-1}$, which has last row $(0\cdots0\,1)$ and determinant $>0$ (because $D$ has positive determinant, and $A_a,B_a$ have determinant of the same sign, for $a\in\{0,1\}$). This holds for every maximal simplex $\Delta_{a_0\ldots a_{t-1}}$ in $\Fcal_t$, and it follows that $\Phi_t$ is orientation-preserving.

Let now $q$ be the vertex in $\Bcal_t$ corresponding to $p$. Again $q$ is in $\Delta^\circ$, and $\Phi(p)=\Phi_t(p)=q$. Let $X=\Delta\setminus\{p\}$ and $Y=\Delta\setminus\{q\}$. Then $\Phi\restriction X$ and $\Phi_t\restriction X$ are homeomorphisms from $X$ to $Y$, and we claim that they are homotopic. Indeed, a homotopy $F:X\times\oi\to Y$ is given by $(x,r)=(1-r)\Phi(x)+r\Phi_t(x)$. This works because, assuming $x\in\Delta_{a_0\ldots a_{t-1}}$, we have by Proposition~\ref{ref10} that $\Phi(x)$ and $\Phi_t(x)$ are both in $\Gamma_{a_0\ldots a_{t-1}}\setminus\{q\}$. Since $\Gamma_{a_0\ldots a_{t-1}}\setminus\{q\}$ is convex, the image of $F$ is $Y$. One checks easily that $F$ is continuous, and this establishes our claim.

Note that, given any points $p',q'\in\Delta^\circ$, the homology groups $H_{n-1}(\Delta\setminus\{p'\})$ and $H_{n-1}(\Delta\setminus\{q'\})$
(coefficients in $\Zbb$) are canonically identifiable, since they are both canonically isomorphic to the relative homology group $H_n(\Delta,\Delta\setminus B)$, where $B$ is any ball in $\Delta^\circ$ containing $p'$ and $q'$.
By~\cite[p.~233]{hatcher02}, we have that $\Phi$ (respectively, $\Phi_t$) is orientation-preserving iff $\Phi\restriction X$ (respectively, $\Phi_t\restriction X$) induces in homology the identity mapping between $H_{n-1}(X)$ and $H_{n-1}(Y)$ (these two infinite cyclic groups canonically identified as above). Since $\Phi\restriction X$ and $\Phi_t\restriction X$ are homotopic, they induce the same isomorphism in homology, and we conclude that $\Phi$ is orientation-preserving iff so is $\Phi_t$.
\end{proof}

As remarked at the beginning of this section, a proper generalization of \S1(2) is critical. Indeed, the periodicity properties of the various multidimensional continued fraction algorithms are a long-standing open problem. Even for the most studied algorithm, the Jacobi-Perron one~\cite{schweiger73}, \cite{schweiger00}, it is still unknown whether points $p=(\vect\alpha1n)$ such that $[\Qbb(\vect\alpha1n):\Qbb]\le n+1$ are always preperiodic under the piecewise-fractional map associated to the algorithm. The situation for the \monk{} algorithm is no better; we list a few simple facts in order to describe the problem.

Let $p\in\Delta$; the \newword{grand orbit} of $p$ under $M$ is
$$
\GO_M(p)=\{q\in\Delta:M^t(p)=M^s(q)\text{ for some $t,s\ge0$}\},
$$
and its \newword{eventual periodic orbit} is
$$
\EPO_M(p)=\{q\in\Delta:q=M^t(p)=M^s(p)\text{ for some $0\le t<s$}\}.
$$
$\EPO_M(p)$ is always a finite set, possibly empty; if it is nonempty, then $p$ is \newword{preperiodic} under $M$. One defines $\GO_T(p)$ and $\EPO_T(p)$ similarly; of course $\Phi\bigl[\GO_M(p)\bigr]=\GO_T\bigl(\Phi(p)\bigr)$ and $\Phi\bigl[\EPO_M(p)\bigr]=\EPO_T\bigl(\Phi(p)\bigr)$.
Let $\dyadic=\{a/2^m\in\Qbb:a,m\in\Zbb\text{ and }m\ge0\}$ be the ring of dyadic rationals. It is a p.i.d., since it is a localization of the p.i.d.~$\Zbb$. 
For $p=(\vect\alpha1n)\in\Delta$, we write $\Qbb(p)$ for the field $\Qbb(\vect\alpha1n)$, and we write $\dyadic(p)$ for the $\dyadic$-module $\dyadic\alpha_1+\cdots+\dyadic\alpha_n+\dyadic$, which is free of rank $\le n+1$.
Since the matrices $M_0,M_1$ determining $M$ are in $\GL_{n+1}\Zbb$, we have clearly $\Qbb(p)=\Qbb(q)$, for any $q\in\GO_M(p)$.
Analogously, the matrices $T_0,T_1\in\Mat_{n+1}\Zbb$ determining $T$ are in $\GL_{n+1}\dyadic$, and hence $\dyadic(p)=\dyadic(q)$, for any $q\in\GO_T(p)$.
We call a point 
$p=(\vect\alpha1n)\in\Delta$ a \newword{rational point} (respectively, a \newword{dyadic point}) if $\vect\alpha1n\in\Qbb$ (respectively,
$\vect\alpha1n\in\dyadic$).
In order to prove that $\Phi$ determines a 1-1 correspondence between the rational points and the dyadic ones, we need two technical lemmas.

Remember that a nonsingular matrix $H=H_{ij}\in\Mat_{n\times n}\Zbb$ is in row Hermite Normal Form (HNF) if it is upper triangular, $H_{jj}>0$ for every $j$, and $0\le H_{ij}<H_{jj}$ for every $1\le i<j$.
Every nonsingular $A\in\Mat_{n\times n}\Zbb$ has a unique HNF (i.e., there exists a unique $H$ in HNF and a ---unique--- $X\in\GL_n\Zbb$ such that $H=XA$)~\cite[\S2.4.2]{cohen93}. In particular, two nonsingular matrices $A,B\in\Mat_{n\times n}\Zbb$ have the same HNF iff there exists $X\in\GL_n\Zbb$ such that $B=XA$; in this case we write $A\sim B$.

\begin{lemma}
Let\label{ref5} $t\ge1$, and let $\vect a0{{t-1}}\in\{0,1\}$. The matrices $T_{a_{t-1}}\cdots T_{a_0}$ and $T_0^t$ have the same HNF.
\end{lemma}
\begin{proof}
The last row of $T_0$ and $T_1$, and hence of all products $T_{a_{t-1}}\cdots T_{a_0}$, is $(0\cdots0\,1)$. Hence we can safely replace $T_a$ with the $n\times n$ matrix $Q_a$ obtained from $T_a$ by removing the last row and the last column. It now suffices to show that $Q_{a_{t-1}}\cdots Q_{a_0}$ and $Q_0^t$ have the same HNF. Direct computation shows that the entries of $Q_0$ are as follows:
$$
(Q_0)_{ij}=\begin{cases}
1, & \text{if $ij=11$, or $ij=1n$, or $i=j+1$;}\\
-1, & \text{if $i\ge2$ and $j=n$;}\\
0, & \text{otherwise.}
\end{cases}
$$
We have $Q_0^n=2E_0$, where $E_0$ is the $n\times n$ identity matrix;
in particular, all powers of $Q_0^n$ commute with everything.
For $1\le r\le n-1$, denote the HNF of $Q_0^r$ by $E_r$; we have explicitly:
$$
(E_r)_{ij}=\begin{cases}
2, & \text{if $i=j\ge n-r+1$;}\\
1, & \text{if $i=j < n-r+1$, or $i<j=n-r+1$;}\\
0, & \text{otherwise.}
\end{cases}
$$
We work by induction on $t$. Denote by $D$ the $n\times n$ diagonal matrix whose entries along the diagonal are $-1,1,\ldots,1$. Since $Q_1=DQ_0$, we always have $Q_{a_0}\sim Q_0$, and the case $t=1$ is settled. By inductive hypothesis, assume $Q_{a_{t-1}}\cdots Q_{a_1}\sim Q_0^{t-1}$. Then $Q_{a_{t-1}}\cdots Q_{a_1}Q_{a_0}\sim 
Q_0^{t-1}Q_{a_0}$, and we claim that $Q_0^{t-1}Q_{a_0}\sim Q_0^t$.
This is immediate if $a_0=0$, so we assume $a_0=1$. Let $t-1=mn+r$, for some $m\ge0$ and $0\le r<n$. Note that $E_rD\sim E_r$; indeed, $E_rD$ is obtainable from $E_r$ by row operations, namely by first forming $DE_r$ and then, if $0<r$, summing to the first row of $DE_r$ the $(n-r+1)$th row. Therefore we have
\begin{multline*}
Q_0^{t-1}Q_1=Q_0^{mn}Q_0^rDQ_0=Q_0^rDQ_0^{mn+1}\sim E_rDQ_0^{mn+1}\sim\\
\sim E_rQ_0^{mn+1}\sim Q_0^rQ_0^{mn+1}=Q_0^t,
\end{multline*}
as claimed.
\end{proof}

\begin{lemma}
Let\label{ref11} $s>0$, let $\vect a0{{s-1}}\in\{0,1\}$, and let $M_*=M_{a_{s-1}}\cdots M_{a_0}$, $T_*=T_{a_{s-1}}\cdots T_{a_0}$.
Then:
\begin{itemize}
\item[(i)] there exists a unique point $q=(\vect\alpha1n)\in\Delta$ such that $(\alpha_1\cdots\alpha_n\,1)^{tr}$ is a right eigenvector for $M_*$ whose corresponding eigenvalue $\xi$ is positive;
we then have $\Qbb(\xi)=\Qbb(\vect\alpha1n)$;
\item[(ii)] an analogous statement holds for $T_*$; in this case $\xi=1$ and $\vect\alpha1n\in\Qbb$.
\end{itemize}
\end{lemma}
\begin{proof}
Recall that we are identifying $\Rbb^n$ with the plane $\{x_{n+1}=1\}$ in $\Rbb^{n+1}$, the latter viewed as a space of column vectors. Accordingly, given a simplex $\Sigma$ in $\Rbb^n$, we write $\Rp\Sigma$ for the polyhedral cone $\{r(\alpha_1\cdots\alpha_n\,1)^{tr}:r\ge0\text{ and }(\vect\alpha1n)\in\Sigma\}$. Let $\abar\in\cantor$ be defined by $a_t=a_{t\pmod{s}}$. Then, for every $k\ge0$, we have $M_*^{-k}[\Rp\Delta]=\Rp\Delta_{\abar\restriction ks}$. By Lemma~\ref{ref7}, the intersection $\bigcap_{k\ge0}\Delta_{\abar\restriction ks}$ is the singleton of a point $q=(\vect\alpha1n)$.  This immediately implies that $M_*^{-1}$ has (up to scalar multiples) a unique eigenvector in $\Rp\Delta$, namely $(\alpha_1\cdots\alpha_n\,1)^{tr}$, whose corresponding eigenvalue $\xi^{-1}$ is positive, and the first statement in~(i) follows. Observe now that $V^{-1}M_*^{-1}V=A_{a_0}\cdots A_{a_{s-1}}$ is a nonnegative matrix. By the Perron-Frobenius theory (see, e.g., \cite[Chapter~III]{gantmacher59}), there exists a permutation matrix $P$ such that
$P^{-1}A_{a_0}\cdots A_{a_{s-1}}P$ has the block form
$$
\begin{pmatrix}
E_1 & 0 & 0 & 0 \\
* & E_2 & 0 & 0 \\
* & * & \ddots & 0 \\
* & * & * & E_r
\end{pmatrix}
$$
with each $E_i$ a nonsingular primitive matrix. The $m\times m$ matrix $E_r$ has a dominant simple eigenvalue $\rho>0$ whose corresponding one-dimensional right eigenspace is spanned by a strictly positive column vector $(\beta_1\cdots\beta_m)^{tr}\in\Qbb(\rho)^m$. Since $M_*^{-1}$ and $A_{a_0}\cdots A_{a_{s-1}}$ are conjugate by $V$, and the $V$-image of the positive orthant $\Rp^{n+1}$ of $\Rbb^{n+1}$ is $\Rp\Delta$, we have from the first part of the proof that 
$A_{a_0}\cdots A_{a_{s-1}}$ has exactly (up to scalar multiples) one eigenvector in the positive orthant whose corresponding eigenvalue is positive. This eigenvector is necessarily $P(0\cdots0\,\beta_1\cdots\beta_m)^{tr}$, and $\rho=\xi^{-1}$. Going back to $\Rp\Delta$, we have that $(\alpha_1\cdots\alpha_n\,1)^{tr}$ is a real multiple of $VP(0\cdots0\,\beta_1\cdots\beta_m)^{tr}$. Hence 
$(\alpha_1\cdots\alpha_n\,1)^{tr}$ is a real multiple of a vector in $\Qbb(\rho)^{n+1}$, and therefore $\vect\alpha1n\in\Qbb(\rho)=\Qbb(\xi)$; since $M_*$ has integer entries, $\xi\in\Qbb(\vect\alpha1n)$. The same proof shows~(ii); in this case $\xi=1$ because the last row of $T_*$ is $(0\cdots0\,1)$.
\end{proof}

Lemma~\ref{ref11}(i) should be compared with~\cite[Theorem~3.1]{brentjes81} and~\cite[Theorem~42]{schweiger00}. In both cases it is proved that a purely periodic point has coordinates in a field of the form $\Qbb(\xi)$, for $\xi$ an eigenvalue of an appropriate periodicity matrix. However, in the first case it is required that $\xi$ has maximal degree $n+1$, while in the second the periodicity matrix is assumed positive. In Lemma~\ref{ref11}(i) we do not require any of these assumptions (of course, the key point here is Lemma~\ref{ref7}, which does not necessarily hold for a generic multidimensional continued fraction algorithm).

\begin{theorem}
Let\label{ref8} $p\in\Delta$. Then:
\begin{itemize}
\item[(i)] $p$ is rational iff $\EPO_M(p)=\{v_1\}$ iff $p$ is a vertex of some $\Fcal_t$;
\item[(ii)] $p$ is dyadic iff $\EPO_T(p)=\{v_1\}$ iff $p$ is a vertex of some $\Bcal_t$.
\end{itemize}
In particular, the set of rational points $\GO_M(v_1)$ is mapped bijectively by $\Phi$ to the set of dyadic points $\GO_T(v_1)$.
Moreover, we have:
\begin{itemize}
\item[(iii)] $[\Qbb(p):\Qbb]\le n+1$ if $p$ is $M$-preperiodic;
\item[(iv)] $p$ is rational iff $p$ is $T$-preperiodic.
\end{itemize}
\end{theorem}
\begin{proof}
By construction, the $M$-counterimage of the set of vertices of $\Fcal_t$ is the set of vertices of $\Fcal_{t+1}$. Moreover, the vertices $\vect v1{{n+1}}$ of $\Delta$ are such that $M(v_1)=v_1$ and $M(v_j)=v_{j-1}$, for $2\le j\le n+1$. Analogous statements hold for $T$, so in (i) and (ii) the equivalence of the second condition with the third is clear.

(i) If $\EPO_M(p)=\{v_1\}$, then $\Qbb(p)=\Qbb(v_1)=\Qbb$, and $p$ is rational. Let $p$ be rational, and let $\vect l1{{n+1}}\in\Zbb$ be its primitive projective coordinates. Let $p=\varphi(\abar)$. Then, for every $t\ge0$, $p\in\Delta_{\abar\restriction t}$ and, since $\Delta_{\abar\restriction t}$ is unimodular, there exist $0\le k_1(t),\ldots,k_{n+1}(t)\in\Zbb$ such that
$$
\begin{pmatrix}
l_1\\
\vdots\\
l_{n+1}
\end{pmatrix}
=VA_{a_0}\cdots A_{a_{t-1}}
\begin{pmatrix}
k_1(t)\\
\vdots\\
k_{n+1}(t)
\end{pmatrix}.
$$
Let $\bigl(c_1(t)\cdots c_{n+1}(t)\bigr)$ be the last row of
$VA_{a_0}\cdots A_{a_{t-1}}$. The reader can easily prove (compare with~\cite[pp.~40-41]{grabiner92}) that $1\le c_1(t)\le\cdots\le c_{n+1}(t)$ and that the sequence $\{c_2(t)\}_{t\ge0}$ is nondecreasing, with limit $\infty$. Let $t$ be such that $l_{n+1}<c_2(t)$. Then we must have $k_1(t)=1$ and $k_2(t)=\cdots=k_{n+1}(t)=0$. In other words, $(l_1\cdots l_{n+1})^{tr}$ is the first column of $VA_{a_0}\cdots A_{a_{t-1}}$, and hence $p$ is a vertex of~$\Fcal_t$.

(ii) If $\EPO_T(p)=\{v_1\}$, then $\dyadic(p)=\dyadic(v_1)=\dyadic$, and $p$ is dyadic. Conversely, let $p=(\vect\alpha1n)\in\dyadic^n$ be dyadic, $p=\upsilon(\abar)$. Choose $m\ge0$ such that $2^mp\in\Zbb^n$. Then, in projective coordinates, we have
$$
T^{mn}(p)=T_{a_{mn-1}}\cdots T_{a_0}
\begin{pmatrix}
\alpha_1\\
\vdots\\
\alpha_n\\
1
\end{pmatrix}.
$$
By Lemma~\ref{ref5}, there exists $X\in\GL_{n+1}\Zbb$ such that
$$
T^{mn}(p)=XT_0^{mn}
\begin{pmatrix}
\alpha_1\\
\vdots\\
\alpha_n\\
1
\end{pmatrix}
=
X
\begin{pmatrix}
2^m & & & \\
& \ddots & & \\
& & 2^m & \\
& & & 1
\end{pmatrix}
\begin{pmatrix}
\alpha_1\\
\vdots\\
\alpha_n\\
1
\end{pmatrix}
\in\Zbb^{n+1}.
$$
Hence $T^{mn}(p)$ is one of the vertices of $\Delta$ and $T^{mn+n}(p)=v_1$.

(iii) Let $(\vect\alpha1n)=q=M^s(q)=\varphi(\abar)\in\EPO_M(p)$ for some $s>0$. Then $(\alpha_1\cdots\alpha_n\,1)^{tr}$ is a right eigenvector for the matrix $M_{a_{s-1}}\cdots M_{a_0}\in\GL_{n+1}\Zbb$. The corresponding eigenvalue $\xi$ is a real algebraic unit of degree $\le n+1$, and by Lemma~\ref{ref11}(i) we have $\Qbb(p)=\Qbb(q)=\Qbb(\xi)$.

(iv) Let $p$ be rational, and let $0<k\in\Zbb$ be such that $kp\in\Zbb^n$. Since $T_0$ and $T_1$ have both integer entries, the forward $T$-orbit of $p$ is contained in $\Delta\cap(k^{-1}\Zbb)^n$, which is finite set; hence $p$ is preperiodic.
The reverse implication is analogous to~(iii), using Lemma~\ref{ref11}(ii).
\end{proof}

Finally, we discuss \S1(3), i.e., the singularity of $\Phi$ w.r.t.~the Lebesgue measure~$\lambda$. We normalize $\lambda$ so that $\lambda(\Delta)=1$. Let $h(\vect x1n)\in L_1(\Delta,\lambda)$ be defined by
$$
h(\vect x1n)=\frac{1}{x_1(x_1-x_2+1)(x_1-x_3+1)\cdots(x_1-x_n+1)},
$$
and let $\mu$ be the probability measure on $\Delta$ induced by the density $h$, properly normalized (for the rest of this paper we are assuming $n\ge2$, since otherwise $\mu$ is infinite):
$$
\mu(A)=\int_A h\,d\lambda \; \Big/\int_\Delta h\,d\lambda.
$$
The \monk\ map $M$ preserves $\mu$ and is ergodic w.r.t.~it~\cite[Theorems 23--24]{schweiger00}. Note that in the above reference $M$
appears as the restriction of the Selmer map to the absorbing $n$-simplex $D$ in~\cite[Theorem~22]{schweiger00}, and the invariant density is $h'(\vect x1n)=\prod_ix_i^{-1}$. We leave to the reader ---as a simple exercise in the calculus of Jacobians--- to check that our $h$ on $\Delta$ is the density corresponding to $h'$ on $D$.

Given an $n$-simplex $\Lambda$ in $\{x_{n+1}=1\}\subset\Rbb^{n+1}$, let $L$ be an $(n+1)\times(n+1)$ real matrix whose columns express the vertices of $\Lambda$ in projective coordinates, and such that the entries $L_{(n+1)1},\ldots,L_{(n+1)(n+1)}$ in the last row are all $>0$. Then one easily computes that
$$
\lambda(\Lambda)=
\frac{\abs{\det(L)}}{L_{(n+1)1}\cdots L_{(n+1)(n+1)}}.
$$
Applying this fact to $L=VB_{a_0}\cdots B_{a_{t-1}}$, we obtain
\begin{equation*}\tag{$*$}
\lambda(\Gamma_{a_0\ldots a_{t-1}})=2^{-t}.
\end{equation*}

Remember that if $\rho:X\to Y$ is a Borel map and $\sigma$ is a Borel probability measure on $X$, then the \newword{push-forward} of $\sigma$ by $\rho$ is the measure $\rho_*\sigma$ on $Y$ defined by $(\rho_*\sigma)(A)=\sigma(\rho^{-1}A)$. If $\beta$ denotes the Bernoulli measure on $\cantor$ obtained by giving $0$ and $1$ equal weight $1/2$, the formula~($*$) implies that $\upsilon_*\beta=\lambda$ (because such an identity holds on the simplexes $\Delta_{\abar\restriction t}$, for $\abar\in\cantor$ and $t\ge0$, and these simplexes generate the Borel sets in $\Delta$).
Since $\upsilon$ induces a conjugation between the shift map $S$ on $\cantor$ and the tent map $T$ on $\Delta$, it follows that $T$ is ergodic w.r.t.~$\lambda$, and hence $M$ is ergodic w.r.t.~$\Phi^{-1}_*\lambda$. Now, $\mu$ and $\Phi^{-1}_*\lambda$ are different (e.g., $(\Phi^{-1}_*\lambda)(\Delta_0)=1/2\not=\mu(\Delta_0)$), and are both ergodic w.r.t.~the same transformation~$M$. Therefore they are mutually singular~\cite[Theorem~6.10(iv)]{walters82}, and there exists a measurable set $A\subseteq\Delta$ such that $\mu(A)=1$ and $\lambda\bigl(\Phi[A]\bigr)=0$. Since $h\ge 1$ on $\Delta$, we have $\mu\ge C\lambda$ for some constant $C>0$. It follows that each of $\mu$ and $\lambda$ is absolutely continuous w.r.t.~the other, and in particular they have the same sets of full measure. We conclude that $\lambda(A)=1$, and $\Phi$ is singular w.r.t.~$\lambda$.

If $p=\varphi(\abar)\in\Delta$, it is natural to look at the limit
\begin{equation*}\tag{$**$}
\lim_{t\to\infty}
\frac{\lambda\bigl(\Phi[\Delta_{\abar\restriction t}]\bigr)}{\lambda(\Delta_{\abar\restriction t})}
\end{equation*}
as an index of the singularity of $\Phi$ at $p$. As we already observed, $\lambda\bigl(\Phi[\Delta_{\abar\restriction t}]\bigr)=2^{-t}$. By the Shannon-McMillan-Breiman Theorem~\cite[\S13]{Billingsley65} we have, for $\mu$-all $p$ (and hence for $\lambda$-all $p$, since $\mu$ and $\lambda$ have the same nullsets), that
\begin{equation*}\tag{$*{*}*$}
\lim_{t\to\infty}
\frac{-\log\mu(\Delta_{\abar\restriction t})}{t}=h_\mu,
\end{equation*}
where $h_\mu$ is the metrical entropy of $M$ w.r.t.~$\mu$. Without loss of generality, we can assume that $p$ is in the topological interior of $\Delta$. For such a $p$, there exist $t_0$ and a constant $C>0$ such that $C\mu(\Delta_{\abar\restriction t})\le\lambda(\Delta_{\abar\restriction t})\le C^{-1}\mu(\Delta_{\abar\restriction t})$, for all $t\ge t_0$.
It follows that in the identity~($*{*}*$) we can safely substitute $\mu$ with $\lambda$. The value $h_\mu$ is explicitly computed in~\cite[\S5.2]{baldwin92a} as follows: if
$$
G(n)=\int_0^1\frac{\bigl[\log(1+s)\bigr]^n}{s}\; ds,
$$
then
$$
h_\mu=\frac{(n+1)\,G(n)}{n\,G(n-1)}.
$$
Taking logarithms in~($**$) we have
\begin{align*}
\lim_{t\to\infty}\bigl[\log\lambda(\Gamma_{\abar\restriction t})-
\log\lambda(\Delta_{\abar\restriction t})\bigr]
&=
\lim_{t\to\infty}\biggl(-\log2-\frac{\log\lambda(\Delta_{\abar\restriction t})}{t}\biggr)t \\
&=
\lim_{t\to\infty}-(\log2-h_\mu)t.
\end{align*}
For $n=2$ we have $h_\mu\sim0.54807\ldots$ and, as shown in~\cite[\S5.2]{baldwin92a}, $h_\mu$ is monotonically increasing with $n$, tending to the limit $\log2\sim0.69314\ldots$ ---which is the topological entropy of $M$ in every dimension--- as $n$ goes to infinity.
We conclude that, for $\lambda$-all $p$ and every $n\ge2$, the limit~($**$) approaches~$0$ exponentially fast.
On the other hand, since $\lim_{n\to\infty}(\log2-h_\mu)=0$, we might loosely say that the singularity of $\Phi$ decreases with the dimension.

\end{document}